\documentclass[12pt]{amsart}
\usepackage{amssymb,latexsym}
\usepackage{enumerate}

\makeatletter 
\@namedef{subjclassname@2010}{%
  \textup{2010} Mathematics Subject Classification}
\makeatother

\newtheorem{thm}{Theorem}[section]
\newtheorem{cor}[thm]{Corollary}
\newtheorem{lem}[thm]{Lemma}

\newtheorem{prop}[thm]{Proposition}

\theoremstyle{definition}

\newtheorem{rem}[thm]{Remark}

\numberwithin{equation}{section}

\def\R{\mathfrak R}

\def\E{\mathfrak E}

\def\R{\mathfrak R}
\def\rit#1{{\mbox{\rm #1}}}
\def\modx#1#2{\equiv#1\hspace{-1mm}\mod #2}
\def\nmodx#1#2{\not\equiv#1\hspace{-1mm}\mod #2}
\def\itemx#1{\item[{\rm(#1)}]}
\def\br#1{\{#1\}}

\begin{document}

\baselineskip=17pt


\title[generalized $\mathbf{\lambda}$ functions]{Generators of modular function fields obtained from generalized lambda functions}
\author[N. Ishii]{Noburo Ishii}
\address{
8-155@Shinomiya-Koganeduka, Yamashina-ku\\
 Kyoto,607-8022,Japan
 }
\email{Noburo.Ishii@ma2.seikyou.ne.jp}
\date{}

\begin{abstract}
We define a modular function which is a generalization of the elliptic modular lambda function. We show this function and the modular invariant function generate the modular function field with respect to the principal congruence subgroup. Further we study its values at imaginary quadratic points. 
\end{abstract}



\maketitle

\section{Introduction}
 For a positive integer $N$, let $\Gamma(N)$ be the principal congruence subgroup of level $N$ of $\rit{SL}_2(\mathbf Z)$, thus,
\[
\Gamma(N)=\left\{\left. \begin{pmatrix} a & b \\ c & d \end{pmatrix}\in \rit{SL}_2(\mathbf Z)~\right |~ a-1\equiv b\equiv c \equiv 0 \mod N \right\}.
\]
We denote by $A(N)$ the modular function field with respect to $\Gamma(N)$.  For an element $\tau$ of the complex upper half plane, we denote by $L_\tau$  the lattice of $\mathbf C$ generated by $1$ and $\tau$ and by $\wp(z;L_\tau)$ the Weierstrass $\wp$-function relative to the lattice $L_\tau$. Let $e_i (i=1,2,3)$ be the 2-division points of the group $\E_\tau=\mathbf C/L_\tau$. The elliptic modular lambda function $\lambda(\tau)$ is defined by
\[
\lambda(\tau)=\frac{\wp(e_1;L_\tau)-\wp(e_3;L_\tau)}{\wp(e_2;L_\tau)-\wp(e_3;L_\tau)}.
\] The function $\lambda$ generates $A(2)$ and is used instead of the modular invariant function $j(\tau)$ to parametrize elliptic curves. Further $2^4\lambda$ is integral over $\mathbf Z[j]$ (see \cite{LA} 18, \S 6).  Note that $e_3=e_1+e_2$. In the case the genus of $A(N)$ is not $0$, thus $N\geq 6$, $A(N)$ has at least two generators. It is well known that $A(N)$ is a Galois extension over $\mathbf C (j)$ with the Galois group $\rit{SL}_2(\mathbf Z)/\{\pm E_2\}\Gamma(N)$, where $E_2$ is a unit matrix. Therefore $A(N)$ is generated by a function over $\mathbf C(j)$. Henceforth let $N\geq 2$. For the group $\E_\tau[N]$ of $N$-division points of $\E_\tau$, there exists an isomorphism $\varphi_\tau$ of the group $\mathbf Z/N\mathbf Z\oplus\mathbf Z/N\mathbf Z$ to $\E_\tau[N]$  given by $\varphi_\tau((r,s))\equiv (r\tau+s)/N \mod  L_\tau$. If $\{Q_1,Q_2\}$ is a basis of $\mathbf Z/N\mathbf Z\oplus\mathbf Z/N\mathbf Z$, then $\{\varphi_\tau(Q_1),\varphi_\tau(Q_2)\}$ is a basis of $\E_\tau[N]$. In this article, we consider a modular function associated with a basis of the group  $\E_\tau[N]$ which is a generalization of $\lambda(\tau)$, defined by
\begin{equation}\label{genlam1}
\Lambda(\tau;Q_1,Q_2)=\frac{\wp(\varphi_\tau(Q_1);L_\tau)-\wp(\varphi_\tau(Q_1+Q_2);L_\tau)}{\wp(\varphi_\tau(Q_2);L_\tau)-\wp(\varphi_\tau(Q_1+Q_2);L_\tau)}.
\end{equation}
 For $N\ne 6$, we shall show that $\Lambda(\tau;Q_1,Q_2)$ generates $A(N)$ over $\mathbf C (j)$. In the case  $N=6$, $\Lambda(\tau;Q_1,Q_2)$ is not a generator  of $A(6)$ over $\mathbf C (j)$, for any basis $\br{Q_1,Q_2}$ (see Remark \ref{remex}).  For $N$, let us define an integer $C_N$ as follows.  Put $C_2=2^4$. Let $N>2$. If $N=p^m$ is a power of a prime number $p$, then put
\[
C_N=\begin{cases}p^2 &\text{ if }p=2,3,\\
                 p   &\text{ if }p>3.
\end{cases}
\]
If $N$ is not a power of a prime number, then put $C_N=1$. We shall show that $C_N\Lambda(\tau;Q_1,Q_2)$ is integral over $\mathbf Z[j]$, and the value of $C_N\Lambda(\tau;Q_1,Q_2)$ at an imaginary quadratic point is an algebraic integer. Further if $N\ne 6$, then it generates a ray class field modulo $N$ over a Hilbert class field.  For the modular subgroups $\Gamma_1(N)$ and $\Gamma_0(N)$, we have obtained similar results by using generalized lambda functions of different types. See Remark \ref{rem41} and for more details, refer to \cite{IN} and \cite{IK}. Throughout this article, we use the following notation:\newline 
For a function $f(\tau)$ and $A=\begin{pmatrix}a&b\\c&d\end{pmatrix}\in\rit{SL}_2(\mathbf Z)$, $f[A]_2$ and $f\circ A$ represent
\[
f[A]_2=f\left(\frac{a\tau+b}{c\tau+d}\right)(c\tau+d)^{-2},~f\circ A=f\left(\frac{a\tau+b}{c\tau+d}\right).
\]
The greatest common divisor of $a,b\in\mathbf Z$ is denoted by $\rit{GCD}(a,b)$. For an integral domain $R$,  $R((q))$ represents the power series ring of a variable $q$ with coefficients in $R$ and $R[[q]]$ is a subring of $R((q))$ of power series of non-negative order. For $f,g\in R((q))$ and a positive integer $m$, the relation $f-g\in q^mR[[q]]$ is denoted by $f\equiv g \mod q^m$.

\section{Auxiliary results}
Let $N$ be an integer greater than $1$.  Put $q=\rit{exp}(2\pi i\tau/N)$ and $\zeta=\exp(2\pi i/N)$. For an integer $x$, let
$\{x\}$ and $\mu (x)$ be the integers defined  by the following conditions:
\[
\begin{split}
&0\le \{x\}\le \frac N2,\quad \mu (x)=\pm 1,\\
&\begin{cases}\mu(x)=1\qquad &\text{if } x\modx {0,N/2}N,\\
             x\equiv \mu (x)\{x\} \mod N\qquad&\text{otherwise.}
\end{cases}
\end{split}
\]

For a pair of integers $(r,s)$ such that $(r,s)\not\equiv (0,0) \mod N$, consider a function 
$$E(\tau;r,s)=\frac 1{(2\pi i)^2}\wp(\frac{r\tau +s}N;L_{\tau})-1/12$$
on the complex upper half plane. Clearly, 
\begin{equation}\label{fund}
\begin{split}
&E(\tau;r+aN,s+bN)=E(\tau;r,s) \text{ for any integers $a,b$},\\
&E(\tau,r,s)=E(\tau,-r,-s),
\end{split}
\end{equation}
since $\wp(z;L_\tau)$ is an even function. It follows that $E(\tau;r,s)$ is a modular form of weight $2$ with respect to $\Gamma (N)$  from the transformation formula:
\begin{equation}\label{transf}
E(\tau;r,s)[A]_2=E(\tau;ar+cs,br+ds),\text{ for } A=\begin{pmatrix}a & b \\ c & d \end{pmatrix}\in \rit{SL}_2(\mathbb Z).
\end{equation}  
  Put $\omega=\zeta^{\mu (r)s}$ and $u=\omega q^{\{r\}}$. From proof of Lemma 1 of \cite{II}, the $q$-expansion of $E(\tau;r,s)$ is obtained as follows:
{\small
\begin{equation}\label{eq1}
E(\tau ;r,s)=
\begin{cases}\displaystyle
\frac{\omega}{(1-\omega)^2}+\sum_{m=1}^{\infty}\sum_{n=1}^{\infty}n(\omega^n+\omega^{-n}-2)q^{mnN}&\text{if }\{r\}=0,\\
\displaystyle\sum_{n=1}^{\infty}n u^n+\displaystyle\sum_{m=1}^{\infty}\sum_{n=1}^{\infty}n(u^n+u^{-n}-2)q^{mnN}&\text{otherwise}.
\end{cases}
\end{equation}
}
Therefore $E(\tau;r,s)\in \mathbf Q(\zeta)[[q]]$. For an integer $\ell$ prime to $N$, let $\sigma_\ell$ be the automorphism of $\mathbf Q(\zeta)$ defined by $\zeta^{\sigma_\ell}=\zeta^\ell$. On a power series $f=\sum_ma_mq^m$ with $a_m\in\mathbf Q(\zeta)$, $\sigma_\ell$ acts  by $f^{\sigma_\ell}=\sum_m a_m^{\sigma_\ell}q^m$. By \eqref{eq1},
\begin{equation}\label{isom}
E(\tau,r,s)^{\sigma_\ell}=E(\tau,r,s\ell).
\end{equation}
If $(r_2,s_2)$ is a pair of integers such that $(r_2,s_2)\nmodx{(0,0),(r_1,s_1),(-r_1,-s_1)}N$, then  $E(\tau;r_1,s_1)-E(\tau;r_2,s_2)$ is not $0$ and has neither zeros nor poles on the complex upper half plane, because the function $\wp(z;L_\tau)-\wp((r_2\tau+s_2)/N;L_\tau)$ has zeros (resp.poles) only at the points $z\equiv \pm (r_2\tau+s_s)/N$ (resp.$0$)$\mod L_\tau$. The next lemma and propositions are required in the following sections.

\begin{lem}\label{lem2} Let $k\in\mathbf Z$ and $\delta=\rit{GCD}(k,N)$. 
\begin{enumerate}
\itemx i For an integer $\ell$, if $\ell$ is divisible by $\delta$, then $(1-\zeta^\ell)/(1-\zeta^k)\in\mathbf Z[\zeta]$.
\itemx {ii} If $N/\delta$ is not a power of a prime number, then $1-\zeta^k$ is a unit.
\end{enumerate}
\end{lem}
\begin{proof} If $\ell$ is divisible by $\delta$, then there exist an integer $m$ such that $\ell\modx{mk}{N}$. 
Therefore $\zeta^\ell=\zeta^{mk}$ and $(1-\zeta^\ell)$ is divisible by $(1-\zeta^k)$. This shows (i). Let $p_i~(i=1,2)$ be distinct prime factors of $N/\delta$. 
Since $N/p_i=\delta (N/\delta p_i)$, $1-\zeta^{N/p_i}$ is divisible by $1-\zeta^\delta$. Therefore $p_i~(i=1,2)$ is divisible by $1-\zeta^\delta$. This implies that $1-\zeta^\delta$ is a unit. Because of $\rit{GCD}(k/\delta,N/\delta)=1$, $1-\zeta^k$ is also a unit .
\end{proof}
The following propositions are immediate results of \eqref{eq1}.
\begin{prop}\label{prop1} Let $(r_i,s_i)~ (i=1,2)$ be as above. Assume that $\{r_1\}\leq \{r_2\}$. Put $\omega_i=\zeta^{\mu (r_i)s_i}$ and $u_i=\omega_i q^{\{r_i\}}$.
\begin{enumerate}
\itemx i If $\{r_1\}\ne 0$, then 
\[
E(\tau;r_1,s_1)-E(\tau;r_2,s_2)\equiv \sum_{n=1}^{N-1} n(u_1^n-u_2^n)+u_1^{-1}q^N-u_2^{-1}q^N \mod q^N.
\]
\itemx{ii} If $\{r_1\}=0$ and $\{r_2\}\ne 0$, then
\[E(\tau;r_1,s_1)-E(\tau;r_2,s_2)\equiv \frac{\omega_1}{(1-\omega_1)^2}-\sum_{n=1}^{N-1} nu_2^n-u_2^{-1}q^N \mod q^N.
\]
\itemx {iii} If $\{r_1\}=\{r_2\}=0$, then 
\[
E(\tau;r_1,s_1)-E(\tau;r_2,s_2)\equiv \frac{(\omega_1-\omega_2)(1-\omega_1\omega_2)}{(1-\omega_1)^2(1-\omega_2)^2}~\mod q^N.
\]
\end{enumerate}
\end{prop}
\begin{prop}\label{prop2} Let the assumption and the notation be the same as in Proposition \ref{prop1}. Then
\[E(\tau;r_1,s_1)-E(\tau;r_2,s_2)=\theta q^{\{r_1\}}(1+qh(q)),\]
where $h(q)\in\mathbf Z[\zeta][[q]]$ and $\theta$ is a non-zero element of $\mathbf Q(\zeta)$ defined as follows.
In the case of $\{r_1\}=\{r_2\}$,
\[
\theta =\begin{cases}\omega_1-\omega_2\quad&\text{if }\{r_1\}\ne 0,N/2,\\
          -\displaystyle\frac{(\omega_1-\omega_2)(1-\omega_1\omega_2)}{\omega_1\omega_2}\quad&\text{if }\{r_1\}=N/2,\\
\displaystyle\frac{(\omega_1-\omega_2)(1-\omega_1\omega_2)}{(1-\omega_1)^2(1-\omega_2)^2}\quad&\text{if }\{r_1\}=0.
\end{cases}
\]
In the case of $\{r_1\}<\{r_2\}$,
\[
\theta =\begin{cases}\displaystyle \omega_1\quad&\text{if }\{r_1\}\ne 0,\\
\displaystyle\frac{\omega_1}{(1-\omega_1)^2}\quad&\text{if }\{r_1\}=0.
\end{cases}
\]
\end{prop}
\section{Generalized lambda functions}
For a basis $\{Q_1,Q_2\}$ of the group $\mathbf Z/N\mathbf Z\oplus\mathbf Z/N\mathbf Z$, let $\Lambda(\tau;Q_1,Q_2)$ be the function defined by \eqref{genlam1}.
 Henceforth, for an integer $k$ prime to $N$, the function $\Lambda(\tau;(1,0),(0,k))$ is denoted by $\Lambda_k(\tau)$ to simplify the notation, thus, 
\begin{equation}\label{lam1}
\begin{split}
\Lambda_k(\tau)&=\frac{\wp(\tau/N;L_\tau)-\wp((\tau+k)/N;L_\tau)}{\wp(k/N;L_\tau)-\wp((\tau+k)/N;L_\tau)}\\
&=\frac{E(\tau,1,0)-E(\tau;1,k)}{E(\tau,0,k)-E(\tau;1,k)}.
\end{split}
\end{equation}

\begin{prop}\label{prop30}
Let $\{Q_1,Q_2\}$ be a basis of the group $\mathbf Z/N\mathbf Z\oplus\mathbf Z/N\mathbf Z$. Then there exist an integer $k$ prime to $N$ and a matrix $A\in\rit{SL}_2(\mathbf Z)$ such that 
\[\Lambda(\tau;Q_1,Q_2)=\Lambda_{k}\circ A.\]
\end{prop}
\begin{proof}
Each basis $\{Q_1,Q_2\}$ of $\mathbf Z/N\mathbf Z\oplus\mathbf Z/N\mathbf Z$ is given by $\{(1,0)B,(0,1)B\}$ for $B\in \rit{GL}_2(\mathbf Z/N\mathbf Z)$. 
It is easy to see that $\displaystyle B\equiv\begin{pmatrix}1&0\\0&k\end{pmatrix}A \mod N$, for an integer $k$ prime to $N$ and $A=\begin{pmatrix}a & b \\ c & d \end{pmatrix}\in \rit{SL}_2(\mathbb Z)$. Therefore $Q_1\equiv (a,b),Q_2\equiv (ck,dk) \mod N$. Since 
\[\Lambda_k(\tau)=\frac{E(\tau,1,0)-E(\tau;1,k)}{E(\tau,0,k)-E(\tau;1,k)},\]
by \eqref{transf}
\[\Lambda_k\circ A=\frac{E(\tau,a,b)-E(\tau;a+ck,b+dk)}{E(\tau,ck,dk)-E(\tau;a+ck,b+dk)}=\Lambda(\tau; Q_1,Q_2).\]
\end{proof}

Let $A(N)_{\mathbf Q(\zeta)}$ be the subfield of $A(N)$ consisted of all modular functions having Fourier coefficients in $\mathbf Q(\zeta)$. 
By \eqref{eq1}, 
\begin{equation}~\label{rationality}
\Lambda(\tau,Q_1,Q_2)\in A(N)_{\mathbf Q(\zeta)}.
\end{equation}
Theorem 3 of Chapter 6 of \cite{LA} shows that $A(N)_{\mathbf Q(\zeta)}$ is a Galois extension over $\mathbf Q(\zeta)(j)$ with Galois group $\rit{SL}_2(\mathbf Z)/\Gamma(N)\{\pm E_2\}$. 
\begin{prop}\label{th1} Let $N\ne 6$ and let $k$ be an integer prime to $N$. Then
 \[A(N)_{\mathbf Q(\zeta)}=\mathbf Q(\zeta)(\Lambda_k,j).\] 
\end{prop}
\begin{proof} 
By \eqref{isom}, $\Lambda_k^{\sigma_\ell}=\Lambda_{k\ell}$. If $A(N)_{\mathbf Q(\zeta)}=\mathbf Q(\zeta)(\Lambda_1,j)$, then we can write $\Lambda_{k^{-1}}=F(\Lambda_1,j)$ for a rational function $F(X,Y)$ of $X$ and $Y$ with coefficients in $\mathbf Q(\zeta)$. By applying $\sigma_k$ to this equality, we have $\Lambda_1=F^{\sigma_k}(\Lambda_k,j)$, and $A(N)_{\mathbf Q(\zeta)}=\mathbf Q(\zeta)(\Lambda_k,j)$. Therefore we have only to prove the assertion in the case $k=1$. Let $k=1$ and $H$ the invariant subgroup of $\Lambda_{1}$ in $\rit{SL}_2(\mathbf Z)$.  Since $\Lambda_1\in A(N)_{\mathbf Q(\zeta)}$, it is sufficient to show $H\subset\Gamma(N)\{\pm E_2\}$. Let $A=\begin{pmatrix}a&b\\ c&d\end{pmatrix}\in H$, thus, $\Lambda_1\circ A=\Lambda_1$. Then by \eqref{transf} and \eqref{lam1},
\begin{equation}\label{eq3x}
\begin{split}(E&(\tau;a,b)-E(\tau;a+c,b+d))(E(\tau,0,1)-E(\tau,1,1))=\\
& (E(\tau,c,d)-E(\tau;a+c,b+d))(E(\tau;1,0)-E(\tau;1,1) ).
\end{split}
\end{equation}
From Proposition~\ref{prop1} it follows:
\begin{equation}\label{series1}
\begin{split}
&E(\tau,0,1)-E(\tau,1,1)\equiv\theta-\zeta q-2\zeta^{2}q^2 \mod q^3,\\
&E(\tau,1,0)-E(\tau,1,1)\equiv (1-\zeta)q+2(1-\zeta^{2})q^2 \mod q^3,
\end{split}
\end{equation}
where $\theta=\zeta/(1-\zeta)^2$.
 By considering the order of $q$-series in the both side of \eqref{eq3x}, it follows from Proposition~\ref{prop2} that
\begin{equation}\label{mineq}
\min(\{a\},\{a+c\})=\min(\{c\},\{a+c\})+1.
\end{equation}
This equality implies that $\br a,\br{a+c}\ne 0$.  At first we shall show that $c\equiv 0\mod N$. Let us assume that $\br{c}\ne 0$. We have three cases: (i) $\{a\}<\{a+c\}$, (ii) $\{a\}>\{a+c\}$, (iii) $\{a\}=\{a+c\}$. Let us consider the case (i). Then $\{c\}=\{a\}-1\ne 0$. Therefore $0<\br a,\br{c}<\br{a+c}\leq N/2$. By comparing the coefficient of $q^{\{a\}}$ of $q$-series in the both side of \eqref{eq3x}, from Proposition \ref{prop2} it follows that
\[\zeta^{\mu(a)b}\theta=\zeta^{\mu(c)d}(1-\zeta).\]
This gives $|1-\zeta|=1$, and $N=6$, which contradicts the assumption. In the case (ii), $\{c\}=\{a+c\}-1$. Therefore $0<\br{c}<\br{a+c}<\br{a}\leq N/2$. An argument similar to that in the case (i) gives that $N=6$. Now we deal with the case (iii).
Put $\{c\}=t$. Then $\{a\}=\{a+c\}=t+1\leq N/2$, and $t\ne 0,N/2$. Since $t\ne 0$, the equality $\{a\}=\{a+c\}$ implies that $c\equiv -2a\mod N,\mu(a)=-\mu(a+c)$. Therefore $t=2\{a\}$ (resp.$N-2\{a\}$) if $2\{a\}\leq N/2$ (resp.$2\{a\}>N/2$). The equality $\{a\}=t+1$ implies that  $t=N-2\{a\}$,thus $N=3t+2$. Hence $N\geq 5$ and $\{a\}\ne N/2$. From comparing the coefficient of $q^{t+1}$ of $q$-series in the both side of \eqref{eq3x}, from Proposition \ref{prop2} it follows that
\begin{equation}\label{xeq1}
\frac{\zeta}{(1-\zeta)^2}(\omega_1-\omega_3)=(1-\zeta)\omega_2,
\end{equation}
where $\omega_1=\zeta^{\mu(a)b},\omega_2=\zeta^{\mu(c)d},\omega_3=\zeta^{\mu(a+c)(b+d)}$. Therefore,
\[\zeta\omega_1\omega_2^{-1}\left(\frac{1-\omega_3\omega_1^{-1}}{1-\zeta}\right)=(1-\zeta)^2.\]
Let $N=5$. Then $(1-\zeta)$ is not a unit but by Lemma \ref{lem2},$\displaystyle\left(\frac{1-\omega_3\omega_1^{-1}}{1-\zeta}\right)$ is $0$ or a unit.
This gives a contradiction. Let $N\geq 6$. Then $t>1$ and noting that $t<N/2-1,2t-1,N-(t+3)$,  the following congruences are obtained from Proposition \ref{prop1}:
\begin{equation}\label{series2}
\begin{split}
&E(\tau;a,b)-E(\tau;a+c,b+d)\equiv (\omega_1-\omega_3)q^{t+1} \mod q^{t+3},\\
&E(\tau;c,d)-E(\tau;a+c,b+d)\equiv \omega_2q^t-\omega_3q^{t+1} \mod q^{t+2},\end{split} 
\end{equation}

Therefore, comparing the coefficient of $q^{t+2}$ of $q$-series in \eqref{eq3x}, we have: 
\begin{equation*}
\zeta(\omega_1-\omega_3)=(1-\zeta)\omega_3-2(1-\zeta^{2})\omega_2.
\end{equation*}
From this, by using \eqref{xeq1}, it follows that $3+\zeta^{2}=\omega_3/\omega_2$. Therefore $|3+\zeta^{2}|=1$. However $|3+\zeta^{2}|>1$. This is a contradiction. Hence we obtain $c\modx 0N$. From \eqref{mineq}, it is deduced that $a\equiv d\modx{\pm 1}N$.
If necessary,by replacing $A$ by $-A$, we can assume that $A=\displaystyle{\begin{pmatrix}1&b\\0&1\end{pmatrix}}$. By \eqref{eq3x}, 
\begin{equation}
\begin{split}(E(\tau;1,b)-&E(\tau;1,b+1))(E(\tau,0,1)-E(\tau,1,1))=\\
& (E(\tau,0,1)-E(\tau;1,b+1))(E(\tau;1,0)-E(\tau;1,1) ).
\end{split}
\end{equation}
By comparing the coefficients of $q$, 
\[(\zeta^b-\zeta^{b+1})\theta=(1-\zeta)\theta.\]
This implies that $\zeta^b=1$. Hence we obtain $A\in\Gamma(N)$.
\end{proof}
\begin{thm}\label{cor31}
Let $\{Q_1,Q_2\}$ be a basis of the group $\mathbf Z/N\mathbf Z\oplus\mathbf Z/N\mathbf Z$.
 Then $A(N)_{\mathbf Q(\zeta)}={\mathbf Q(\zeta)}(\Lambda(\tau;Q_1,Q_2),j)$.
\end{thm}
\begin{proof}
By Proposition \ref{prop30}, there exists an integer $k$ prime to $N$ and an element $A\in \rit{SL}_2(\mathbf Z)$ such that $\Lambda(\tau;Q_1,Q_2)=\Lambda_k\circ A$.  Since $\Gamma (N)$ is a normal subgroup of $\rit{SL}_2(\mathbf Z)$, the assertion is deduced from \eqref{rationality} and Proposition \ref{th1}.
\end{proof}
\begin{rem}\label{remex}
Let $N=6$. Then the matrix $M=\begin{pmatrix}3&11\\1&4\end{pmatrix}\not\in\Gamma(6)$ fixes the function $\Lambda_1(\tau)$. This fact is proved as follows. Let us consider the function 
\[
\begin{split}
F(\tau)&=(E(\tau,1,0)-E(\tau,1,1))[M]_2(E(\tau,0,1)-E(\tau,1,1))-\\
&\phantom{aaa}(E(\tau,0,1)-E(\tau,1,1))[M]_2(E(\tau,1,0)-E(\tau,1,1))\\
&=(E(\tau,3,1)-E(\tau,2,3))(E(\tau,0,1)-E(\tau,1,1))-\\
&\phantom{aaa}(E(\tau,1,4)-E(\tau,2,3))(E(\tau,1,0)-E(\tau,1,1)).
\end{split}
\]
Here we used \eqref{fund} and \eqref{transf}. Then $F$ is a cusp form of weight 4 with respect to $\Gamma(6)$. If $F\ne 0$, then $F$ has $24$ zeros in the fundamental domain. See \cite{ogg}, III-6, Proposition 10.  Let $A=\begin{pmatrix}a&b\\c&d\end{pmatrix}\in\rit{SL}_2(\mathbf Z)$. Then the order of $F$ at the cusp $a/c=A(i\infty)$ is greater than or equal to minimum of two integers $\min(\{3a+c\},\br{2a+3c})+\min(\br c,\br{a+c})$ and $\min(\br{a+4c},\br{2a+3c})+\min(\br{a},\br{a+c})$.
It is easy to see that $F$ has at least $22$ zeros at cusps other than~ $i\infty$ and the coefficient of $q^2$ of the $q$-expansion of $F$ is $0$. This shows that $F$ has at least $25$ zeros. Hence $F=0$. 
\end{rem}
\section{Values of $\Lambda(\tau;Q_1,Q_2)$ at imaginary quadratic points}
In this section, we study values of $\Lambda(\tau;Q_1,Q_2)$ at imaginary quadratic points. In the case $N=2$, it is a well known that $2^4\lambda$ is integral over $\mathbf Z[j]$. For example see \cite{LA} 18, \S 6. We shall consider the case $N>2$.
\begin{lem}\label{lem3a} Let $k$ be an integer prime to $N$ and $A\in\rit{SL}_2(\mathbf Z)$. Let $A_k$ be a matrix of $\rit{SL}_2(\mathbf Z)$ such that 
$\displaystyle A_k\equiv \begin{pmatrix}a&bk^{-1}\\ck&d\end{pmatrix}\mod N$.
Then \[\Lambda_k\circ A=(\Lambda_1\circ A_k)^{\sigma_k}.\]
\end{lem}
\begin{proof}
Let $A_k=\begin{pmatrix}t&u\\v&w\end{pmatrix}$. Then
\[
\begin{split}
(\Lambda_1\circ A_k)^{\sigma_k}&=\frac{E(t,uk)-E(t+v,(u+w)k)}{E(v,wk)-E(t+v,(u+w)k)}\\
&=\frac{E(a,b)-E(a+ck,b+dk)}{E(ck,dk)-E(a+ck,b+dk)}\\
&=\Lambda_k\circ A.
\end{split}
\]
\end{proof}
\begin{prop}\label{prop3a} Let $N>2$ and $k$ be an integer prime to $N$. Then for any $A\in\rit{SL}_2(\mathbf Z)$, $(1-\zeta^k)^3\Lambda_k\circ A\in\mathbf Z[\zeta]((q))$.
\end{prop}
\begin{proof} 
By Lemma \ref{lem3a}, we have only to prove the assertion in the case $k=1$.
Put $A=\begin{pmatrix}a&b\\c&d\end{pmatrix}$. Proposition \ref{prop2} shows that
\begin{equation}
\begin{split}&E(\tau;a,b)-E(\tau;a+c,b+d)=\theta_1q^{t_1}(1+h_1(q)),\\
& E(\tau,c,d)-E(\tau;a+c,b+d)=\theta_2q^{t_2}(1+h_2(q)),
\end{split}
\end{equation}
where $t_i$ are non-negative integers, $\theta_i$ are non-zero elements of $\mathbf Q(\zeta)$ and $h_i\in\mathbf Z[\zeta][[q]]$ ($i=1,2$). This shows $\Lambda_{k}\circ A=\omega f(q)$, where $\omega=\theta_1/\theta_2$ and $f\in\mathbf Z[\zeta]((q))$. Therefore it is sufficient to prove that $(1-\zeta)^3\omega\in \mathbf Z[\zeta]$. By Proposition \ref{prop2}, if $\min(\br a,\br{a+c})\ne 0$ and $\br{c}\ne\br{a+c}$, then $\theta_1,\theta_2^{-1}\in \mathbf Z[\zeta]$. Therefore $\omega\in\mathbf Z[\zeta]$. Let $\br{c}=\br{a+c}$. If $\mu(c)=\mu(a+c)$,then $a\modx 0N$. This implies that $\rit{GCD}(c,N)=1$ and $\br a=0<\br{c}=\br{a+c}<N/2$. Therefore 
\[
\theta_1=\zeta^b/(1-\zeta^b)^2, \theta_2=\zeta^{\mu(c)d}-\zeta^{\mu(c)(b+d)},\]
and $\omega=\zeta^\ell/(1-\zeta^b)^3$ for an integer $\ell$. Since $\rit{GCD}(b,N)=1$, by (i) of Lemma \ref{lem2}, $(1-\zeta)^3\omega\in\mathbf Z[\zeta]$. Let $\mu(c)=-\mu(a+c)$. Then $a\modx {-2c}N$. Since $\rit{GCD}(a,c)=1$, $\rit{GCD}(c,N)=1$. It follows that $\{c\}\ne 0,N/2$ and $\br a,\br{a+c}\ne 0$. Therefore $\theta_1\in\mathbf Z[\zeta]$ and $\theta_2=\zeta^{\mu(c)d}(1-\zeta^{-\mu(c)(b+2d)})$. Let $\rit{GCD}(b+2d,N)=D$, then $b\modx{-2d}D,a\modx{-2c}D$. It follows that $1=ad-bc\modx 0D$. This shows $b+2d$ is prime to $N$. Lemma~\ref{lem2} shows that $(1-\zeta)\omega\in\mathbf Z[\zeta]$. Let $\min(\br a,\br{a+c})=0$ and $\br{a+c}\ne \br{c}$. Then $\br{a+c}=0$ and $\br a, \br{c}\ne 0$. Therefore $0=\br{a+c}<\br a,\br{c}$, and $\theta_1=\theta_2$, thus $\omega=1$.\end{proof} 
Let $C_2=2^4$ and for $N>2$ put
\[
\begin{split}
C_N=\begin{cases} p^2&\text{ if }N=p^\ell (~p=2,3),\\
                  p&\text{ if }N=p^\ell (~p:\text{a prime number}>3),\\
                  1&\text{ if $N$ is not a power of a prime number}.
\end{cases}
\end{split}
\]

\begin{cor}\label{cor3a}
Let $N>2$ and $k$ be an integer prime to $N$. 
Then $C_N\Lambda_k\circ A\in\mathbf Z[\zeta]((q))$ for any $A\in\rit{SL}_2(\mathbf Z)$.
\end{cor}
\begin{proof}
 Lemma~\ref{lem2} implies that $C_N/(1-\zeta^k)^3\in\mathbf Z[\zeta]$. The assertion follows from Proposition \ref{prop3a}. 
\end{proof}

\begin{thm}\label{thm3a}
 Let $\{Q_1,Q_2\}$ be a basis of the group $\mathbf Z/N\mathbf Z\oplus\mathbf Z/N\mathbf Z$. Then the function $C_N\Lambda(\tau;Q_1,Q_2)$ is integral over $\mathbf Z[j]$. Further Let $\theta$ be an element of the complex upper half plane such that $\mathbf Q(\theta)$ is an imaginary quadratic field. Then $C_N\Lambda(\theta;Q_1,Q_2)$ is an algebraic integer.
\end{thm}
\begin{proof}
For $N=2$, the assertion has already proved. Let $N>2$. For an integer $k$ prime to $N$, let us consider a polynomial of $X$:
\[
\Psi_k(X)=\prod_{A}(X-C_N\Lambda_k\circ A),
\]
where $A$ runs over all representatives of $\rit{SL}_2(\mathbf Z)/\Gamma(N)\{\pm E_2\}$. Then each coefficient of $\Psi_k(X)$ is belong to $\mathbf Z[\zeta]((q))$ and is $\rit{SL}_2(\mathbf Z)$-invariant, and has no poles in the complex half plane. Therefore $\Psi_k(X)$ is a monic polynomial with coefficients in $\mathbf Z[\zeta][j]$. Since $C_N\Lambda_k\circ A$ is a root of $\Psi_k(X)=0$,  $C_N\Lambda_k\circ A$ is integral over $\mathbf Z[\zeta][j]$. From Proposition \ref{prop30} and the fact that $\mathbf Z[\zeta][j]$ is integral over $\mathbf Z[j]$, it follows that $C_N\Lambda(\tau;Q_1,Q_2)$ is integral over $\mathbf Z[j]$. Since $j(\theta)$ is an algebraic integer (see \cite{C1},Theorem 10.23) and $C_N\Lambda(\theta;Q_1,Q_2)$ is integral over $\mathbf Z[j(\theta)]$, $C_N\Lambda(\theta;Q_1,Q_2)$ is an algebraic integer. 
\end{proof}
\begin{thm}\label{th4b}
Let $N\ne 6$ and $\{Q_1,Q_2\}$ be a basis of the group $\mathbf Z/N\mathbf Z\oplus\mathbf Z/N\mathbf Z$. Let $\theta$ be an element of the complex upper half plane such that $\mathbf Z[\theta]$ is the maximal order of an imaginary quadratic field $K$. Then  the ray class field $\R_N$ of $K$ modulo $N$ is generated by $\Lambda(\theta;Q_1,Q_2)$ and $\zeta$ over the Hilbert class field $K(j(\theta))$ of $K$.
\end{thm} 
\begin{proof}
The assertion is deduced from Theorems 1 and 2 of \cite{GA} and Theorem \ref{cor31}.
\end{proof}
\begin{rem}\label{rem41}
 Let $k$ and $\ell$ be integers such that $0<k\ne \ell<N/2,\rit{GCD}(k+\ell,N)=1$. We consider a function
\[
\Lambda_{k,\ell}^*(\tau)=\frac{\wp(\frac k N;L_\tau)-\wp(\frac{k+\ell}N;L_\tau)}{\wp(\frac \ell N;L_\tau)-\wp(\frac{k+\ell}N;L_\tau)}.  
\]
This is a modular function with respect to the group
\[
\Gamma_1(N)=\left\{\left. \begin{pmatrix} a & b \\ c & d \end{pmatrix}\in \rit{SL}_2(\mathbf Z)~\right |~ a-1\equiv c \equiv 0 \mod N \right\}.
\]
 In Corollary 1 of \cite{IN} we show that $\Lambda_{k,\ell}^*$ and $j$ generate the function field rational over $\mathbf Q(\zeta)$ with respect to $\Gamma_1(N)$ . Let the notation be the same as in Theorem \ref{th4b}. From Corollary 3  and  Theorem 4 of \cite{IN}, we obtain that $\R_N$ is generated by $\Lambda_{k,\ell}^*(\theta)$ and $\zeta$ over the Hilbert class field of $K$ and that $\Lambda_{k,\ell}^*(\theta)$ is an algebraic integer under an additional assumption $\rit{GCD}(k(k+2\ell),N)=1$.
\end{rem}

\end{document}